  \documentclass[a4, 12pt]{article}
  \usepackage{amsmath, amssymb, latexsym, amscd, amsthm,amstext}
  \usepackage{authblk}
  	\usepackage{hyperref}                            % For creating hyperlinks in cross references  
 	 \usepackage{xcolor}
	 \usepackage{cite}
  	\usepackage{enumerate}  
  	\usepackage{multicol}   
	\usepackage{wrapfig}  
	\usepackage{graphicx}
  \def\R{\mathbb{ R}}
  
  \def\N{\mathbb{ N}}

  \def\Sy{\mathbb{ S}}

  \def\rk{\mbox{\rm \texttt{rank}}}
  \def\tr{\mbox{\rm \texttt{Tr}}}
  
  \def\diag{\mbox{\rm \texttt{diag}}}

  \def\vect{\mbox{\rm \texttt{vec}}}

  \def\vect{\mbox{\rm \texttt{vec}}}
  \def\svec{\mbox{\rm \texttt{svec}}}
  \def\sp{\mbox{\rm \texttt{Span}}}
  \def\jac{\mbox{\rm \texttt{Jac}}}  
  
  \newtheorem{prop}{\bf Proposition}
  
  \newtheorem{rema}{\it Remark}
  \newtheorem{alg}{\bf Algorithm}

  \rmfamily
  
\begin{document}
 \title{\bf 
Low rank
solutions to  differentiable systems over matrices
and applications\thanks{
The research is supported by Vietnam National Foundation for Science and Technology Development (NAFOSTED) under grant number 
		101.01-2014.30.
	\newline
 	Email: \texttt{lethanhhieu@qnu.edu.vn}			
		}	
}

\author%[a]
{Thanh Hieu LE
}
\affil%[a]
{Department of Mathematics, Quy Nhon University, Vietnam}

%\author[b]{Marc Van Barel\thanks{
%		The research of the second author was partially supported by
%%
%  the Research Council KU Leuven,
%PF/10/002 Optimization in Engineering Centre (OPTEC),
%%
%by
%%
%the Fund for Scientific Research--Flanders (Belgium),
%%  projects G.0078.01 (SMA: Structured Matrices and their Applications),
%%  G.0176.02 (ANCILA: Asymptotic aNalysis of the Convergence behavior of Iterative methods in numerical Linear Algebra),
%%  G.0184.02 (CORFU: Constructive study of Orthogonal Functions)
%%   G.0455.0 (RHPH: Riemann-Hilbert problems, random matrices and Pad\'e-Hermite approximation),
%%   G.0423.05 (RAM: Rational modelling: optimal conditioning and stable algorithms),
%G.0828.14N (Multivariate polynomial and rational interpolation and approximation),
%%
%   and by
%   the Interuniversity Attraction Poles Programme, initiated by the Belgian State,  Science Policy Office,
%   Belgian Network DYSCO (Dynamical Systems, Control, and Optimization).
%The scientific responsibility rests with its author(s).
%\\
%\quad Email: \texttt{marc.vanbarel@cs.kuleuven.be} }}
%\affil[b]{Department of Computer Science, KU Leuven, Belgium}

%\date{January 13, 2017}
\date{} % Update 23/2/2017
 \maketitle
 
 \begin{abstract}
Differentiable systems in this paper means systems of equations that are described by differentiable real functions in real matrix variables.
 This paper proposes algorithms for 
finding minimal rank solutions to 
such
systems over (arbitrary and/or several structured) matrices
by using the Levenberg-Marquardt method (LM-method) for solving least squares problems. 
We then  
apply these algorithms
to solve several engineering problems such as 
the low-rank matrix completion problem and
the low-dimensional Euclidean embedding one.
Some numerical experiments illustrate the validity of the approach.

On the other hand, 
we provide some further properties of low rank solutions to
systems linear matrix equations.
This is useful when the differentiable function is linear or quadratic.
 \end{abstract}


%%
%% Start line numbering here if you want
%%
% \linenumbers

%% main text

\section{Motivation and preliminaries}

%\subsection{Motivation}\label{rmp_intr}

Several problems in either engineering or computational  mathematics
can be reformulated as 
{\it rank minimization problems} 
(shortly, RM-problems) in the form
\begin{equation}\label{eq_grmp_prob}
\begin{array}{llll}
\mbox{minimize} & \rk(X) \\
\mbox{subject to}& \\
								& X\in \mathcal{C}, 
\end{array}
\end{equation}
where $\mathcal{C}$ is a subset of $\R^{m\times n},$
the set of all $m$ by $n$ matrices with real entries.

RM-problem (\ref{eq_grmp_prob}) is computationally NP-hard in general,
even when $\mathcal{C}$ is an affine subset of 
$\R^{m\times n}.$
There hence is a number of algorithms for solving this problem with respect to special cases of $\mathcal{C},$ 
see, eg.,
 \cite{RFP10,Fazel02,LeSoVa13} and the references there in.
When the constraints are  defined by linear matrix equations, 
i.e.,
$\mathcal{C}$ is the solution set of a linear system of equations $\ell(X) =b\in \R^k,$
the present problem is called
{\it affine rank minimization problem} 
(shortly, ARM-problem)
and is in the form 
\cite{RFP10} 
\begin{equation}\label{eq_rmp_prob}
\begin{array}{llll}
\mbox{minimize} & \rk(X) \\
\mbox{subject to}& \\
								&X\in \R^{m\times n},\\
								& \ell(X) = b.
\end{array}
\end{equation}
When the constraint region is considered on the cone of positive semidefinite matrices,  
the authors in  \cite{MP97}
relaxed the non-convex rank objective function in problem (\ref{eq_rmp_prob}) into the nuclear 
norm that is a convex function.
The whole problem is then a semidefinite program \cite{p878}
and can be efficiently solved by SDP solvers.
%
%It is shown that such a problem
%can be solved efficiently, and is the best 
%lower approximation of the rank function over the set of matrices with spectral norm
%less than or equal to one.
%Indeed,
%let $\|X\|_*$ and $\|X\|$ denote the nuclear norm (i.e., the sum of singular values of $X$) 
%and  spectral  norm of matrices (i.e., the maximal singular value of $X$), respectively.
%Then for any matrix $X$ of rank $r,$ one always has
%$$
%\| X\| \leq \|X\|_*  \leq r \|X\|.
%$$
%So
%$\rk(X) \geq \| X\|_*/\|X\|.$ 
%This yields $\|X\|_* \leq \rk(X)$ for all $X$ with $\|X\|\leq 1.$
%Moreover, if $X_0$ is a minimum rank solution to (\ref{eq_grmp_prob}) over 
%$\mathcal{C} = \{X\in \R^{m\times n}:\  \ell(X)=b\},$
%and $X_* \in \mathcal{C}$  with minimal nuclear norm
%then
%$$
%\|X_*\|_*/\|X_0\| \leq \rk(X_0) \leq \rk(X_*).
%$$
%
%
In our point of view, 
by using the Cholesky decomposition, 
each positive semidefinite matrix $X$ can be written as $X = YY^T,$
$Y\in \R^{n\times n}.$
The linear map in the later problem (\ref{eq_rmp_prob})  now becomes a quadratic map in $Y.$

%What we have discussed suggests us a more general class of problem where the linear/quadratic maps is generalized to a differential one.
In this paper, 
we focus on the problem over a more general set $\mathcal{C},$ in comparison with the sets we have discussed above.
Such a set is determined by a differentiable map.
That is, we focus on the problem
\begin{equation}\label{eq_Frmp_prob}
\begin{array}{llll}
\mbox{minimize} & \rk(X) \\
\mbox{subject to}& \\
								&X\in \mathcal{C},\\
								& \phi(X) = b,
\end{array}
\end{equation} 
\noindent
where 
$\mathcal{C} \subseteq \R^{m\times n}$
and
$\phi: \R^{m\times n} \rightarrow \R^k$ is a differentiable  map.
This function is clearly non-convex in general.
Our method applies the generalized Levenberg-Marquardt method \cite{r665} for checking whether there exists a solution of rank $r,$
step by step, for $r=1,2,\ldots$
The differentiability of $\phi$ guarantees for the existence of its Jacobian in the Levenberg-Marquardt steps. 
It turns out that the problem of 
finding a matrix of rank $r = 1, 2, \ldots, \min\{m,n\},$ solving the equation $\phi(X)=b$
is the most important in our method.

%\subsection{Preliminaries} \label{sec_pre}

We now recall some important results on matrix factorization in linear algebra that are used in the paper.

By $.^\mathrm{T}$ we denote the transpose  of matrices.
For a real symmetric matrix $A,$
i.e., $A^\mathrm{T} =A,$ 
by
$A\succeq 0$ we mean $A$ is positive semidefinite, i.e., $x^\mathrm{T}Ax \geq 0$ for all $x\in \R^n.$ 
This, equivalently, means its eigenvalues  are all non-negative.
For any two real symmetric matrices $A$ and $B,$
we write $A\succeq B$ if $A-B\succeq 0.$
Let $\Sy^n$ denote the set of $n$ by $n$ real symmetric  matrices,
and $\Sy_+^n$  denote the cone of  positive semidefinite matrices in $\Sy^n.$

\begin{prop}{\rm(see, e.g, \cite{Bhatia07} or \cite[Section 2.6, Observation 7.1.6]{b398})}\label{prop_ChoFac}
Any positive semidefinite matrix (PSD matrix) $A\in \Sy_+^n$ has a Cholesky decomposition 
$A = LL^\mathrm{T},$
where 
$L\in \R^{n\times n}$ 
is a lower triangular matrix which is called a Cholesky factor of $A.$
In particular,
if 
$r=\rk(A)$
then 
one can find $L\in \R^{n\times r}.$

Another fact is that
for two matrices $A,B \in \Sy^n,$ then
$A\succeq B$
if and only if
$P^{T}AP \succeq P^{T}BP$
for any nonsingular matrix $P\in \R^{n\times n}.$
%\red{The notation $\succeq$ should be defined.}
\end{prop}

\begin{prop}{\rm (see e.g, \cite[Section 0.4.6]{b398}}
\label{prop_rkMat}
Let $A$ be an $m\times n$ real matrix. Then 
\begin{enumerate}[\rm i)]
\item
	$
	\rk(A) = \rk(A^\mathrm{T}) = \rk(A A^\mathrm{T}) = \rk(A^\mathrm{T}A).
	$
\item
	$A\in \R^{m\times n}$ has rank $r$ if and only if there exist matrices 
$X\in \R^{r\times m},$ 
$Y\in \R^{r\times n}$
and 
$B\in \R^{r\times r}$
nonsingular  
with 
$\rk(X) = \rk(Y)=r$
such that 
$A = X^\mathrm{T} BY.$

A consequence,
$A$ can be written as 
$A= X^T Z$
with 
$Z= BY \in \R^{r\times n}$
and $\rk(X)= \rk(Z)=r.$
\end{enumerate}

\end{prop}

\begin{prop}\label{prop_CharLmap}
For any linear map 
$\ell: \R^{m\times n} \rightarrow \R$ 
one can find a matrix $A\in \R^{m\times n}$ such that
$$
\ell(X)= \tr(A^\mathrm{T}X) = \tr(AX^\mathrm{T}) , \quad\forall X\in \R^{m\times n}.
$$

Specially, 
if $\ell: \Sy^n \rightarrow \R$  
then  $A$ can be found in $\Sy^n,$
i.e., $A^\mathrm{T}=A.$ 
\end{prop}
\begin{proof}
Suppose  $\ell: \R^{m\times n} \rightarrow \R$ 
is a linear map.
Consider  $\R^{m\times n}$ as a real vector space endowed with the basis
$\{E_{ij}| \ i=1,\ldots, m; \ j=1,\ldots,n \},$
where $E_{ij}$ is the $m\times n$ matrix whose entries are zeros except for
the $(i,j)$th one being 1. 
Let 
$A = [\ell(E_{ij})]_{i=1,\ldots,m}^{j=1,\ldots,n} \in \R^{m\times n}.$
Then for every 
$X= [x_{ij}] \in \R^{m\times n},$
$X = \sum_{i,j} x_{ij}E_{ij},	$
we have
$$
\ell(X)  
= \sum_{i=1}^m \sum_{j=1}^n x_{ij} \ell(E_{ij})
= \tr(AX^\mathrm{T})
= \tr(A^\mathrm{T}X).
$$
 
If $\ell: \Sy^n \rightarrow \R$ then it follows that
\begin{align*}
\ell(X) 
&= \sum_{i=1}^m \sum_{j=1}^n x_{ij} \ell(E_{ij})
= \sum_{i=1}^n x_{ii}\ell(E_{ii})  + \sum_{1\leq i<j \leq n} x_{ij}[\ell(E_{ij}) + \ell(E_{ji})] \\
&= \tr[(\frac{A^\mathrm{T}+A}{2})X],
\end{align*}
and the proof is done.
\end{proof}

%\subsection{Derivative of functions}
 We now recall some notation and results from matrix calculus. 
 Let $f:\R^n \rightarrow \R^m$ be a $m\times 1$ vector function of a $n\times 1$ vector $x.$
 The  {\it derivative} (or {\it Jacobian matrix}) of $f$ is the $m\times n$ matrix defined by
 $$
 \jac f(x) \ \triangleq\ \frac{\partial f(x)}{\partial x} 
 = 
 \begin{bmatrix}
 \frac{\partial f_1(x)}{\partial x_1}  & \ldots   & \frac{\partial f_1(x)}{\partial x_n} \\
 \vdots                                                  & \ddots   & \vdots \\
 \frac{\partial f_m(x)}{\partial x_1} & \ldots    & \frac{\partial f_m(x)}{\partial x_n}
 \end{bmatrix} 
 \in \R^{m\times n}.
 $$
 
 We now recall a general definition for the derivative of a {\it matrix valued  function}. 
 Suppose 
 $F: \R^{m\times n} \rightarrow \R^{p\times q}$ 
 is a  $(p\times q)$-matrix valued function of an $(m\times n)$-matrix variable $X.$
 Suppose that $F= [F_{rs}] \in \R^{p\times q}$
 and 
we define the derivative of this function as the $pq\times mn$ matrix 
 %SEE WIKIPEDIA ``matrix calculus''
 $$
 \jac F(X) 
 \triangleq
\frac{\partial\vect{F}(X)}{\partial \vect X}
= 
\begin{bmatrix}
\frac{\partial F_{11}(X)}{\partial x_{11}} & \frac{\partial F_{11}(X)}{\partial x_{21}} & \ldots & \frac{\partial F_{11}(X)}{\partial x_{mn}}\\
\frac{\partial F_{21}(X)}{\partial x_{11}} & \frac{\partial F_{21}(X)}{\partial x_{21}} & \ldots & \frac{\partial F_{21}(X)}{\partial x_{mn}}\\
				\vdots								 & 						\vdots						  & \ddots & \vdots\\
\frac{\partial F_{pq}(X)}{\partial x_{11}} & \frac{\partial F_{pq}(X)}{\partial x_{21}} & \ldots & \frac{\partial F_{pq}(X)}{\partial x_{mn}}\\
\end{bmatrix}
\in \R^{pq\times mn}, 
 $$ 
 where
 $\vect X \in \R^{mn\times 1}$ 
 denotes the vector obtained  by stacking its columns one underneath
the other,
 i.e.,
 if $X\in \R^{m\times n}$ and $X_j, \ j=1,\ldots,n,$ are the columns of $X$ then
 $$
 \vect X = 
 \begin{bmatrix}
 X_1^\mathrm{T} &  \ldots & X_n^\mathrm{T}
 \end{bmatrix}^\mathrm{T}.
 $$
% Note that each column of 
% $\mathrm{D}F(X)$ 
% is the $\vect$-operator applied to the
% $(p\times q)$ 
% matrix
% $\frac{\partial F}{\partial x_{i j}}.$
% %
%The $\vect$-operator has some properties that are used later as follows \cite[Section 10.2.2]{q715}:
%%\red{Are all these properties used in the sequel?}
%\begin{align}
%%\vect(AXB) &= (B^\mathrm{T}\otimes A)\vect(X), \\
%\tr(A^\mathrm{T}X)  &= \vect(A)^\mathrm{T}\vect(X), \\
%\vect(A+B) &= \vect(A) +\vect(B),\\
%\vect(\alpha A) &= \alpha \vect(A),\\
%\mathbf{a}^\mathrm{T}XBX^\mathrm{T}\mathbf{c} &= \vect(X)^\mathrm{T} (B\otimes \mathbf{c} \mathbf{a}^\mathrm{T}) \vect(X). 
%\end{align}
% 
 We list below the important properties of the derivative of trace functions  
 that will be used in either paper or \textsc{Matlab} codes (see, eg., \cite{q715}).
 \begin{itemize}
 \item
 	Let $A$ be a given matrix in $\R^{m\times n}.$ 
 			If $F: \R^{m\times n} \rightarrow \R$ is defined by 
 			$F(X) = \tr(A^\mathrm{T}X),\ \forall X\in \R^{m\times n},$
 			then 
			\begin{equation}
 			\jac F(X) = \vect(A)^\mathrm{T}.			
			\end{equation}
 \item 
 	Let $A\in \R^{n\times n}$ and  $B\in \R^{m\times m}$ be given.
 			If $F: \R^{m\times n} \rightarrow \R$ is defined by 
 			$F(X) = \tr(XAX^\mathrm{T}B)$
 			then 
 			\begin{equation}\label{eq_JacTr1}
 			\jac F(X) = \vect(B^\mathrm{T}XA^\mathrm{T} + BXA)^\mathrm{T}. 			
 			\end{equation}
 			
\item
 	Let $A,B$ be two given matrices in $\R^{m\times m}.$ 
 			Then for all $X\in \R^{m\times m},$
			\begin{equation}
 			\jac \tr(AXB) = \vect(A^\mathrm{T}B^\mathrm{T})^\mathrm{T}
 			\quad
 			\mbox{and}
 			\quad
 			\jac \tr(AX^\mathrm{T}B) = \vect(BA)^\mathrm{T}.				
			\end{equation}
 \end{itemize}
  
This paper is organized as follows.
Section \ref{sec_idea} 
presents the main algorithm for solving problem (\ref{eq_Frmp_prob}).
This algorithm will be applied to particular problems with respect to several types of constraint sets. 
The affine rank minimization problem over arbitrary matrices is presented in Section  \ref{sec_rmp_on_ArbMat}
and a similar method applied for positive semidefinite matrices is handled in Section \ref{sec_rmp_on_psdMat}.
Section \ref{sec_app} summarizes some applications  
of our solution method to several problems in engineering.
The corresponding 
numerical experiments are exhibited in Section \ref{sec_numexp}.
The last section presents the conclusion and discussion for the future works.

\section{The idea for solving problem (\ref{eq_Frmp_prob}) } \label{sec_idea}

In this work,
with the help of Proposition \ref{prop_rkMat},
we solve problem (\ref{eq_Frmp_prob}) by 
using the generalized Levenberg-Marquardt method \cite{r665} to
find a matrix $X \in \mathcal{C}\subset \R^{m\times n}$ step by step for 
$\rk(X)= 1, 2, \ldots,\min\{m,n\}$ such that 
$\phi(X) = b.$  
In this situation,
we consider the least square problem 
with respect to
the function $F: \R^{\mu}  \rightarrow  \R^k,$
with appropriate integer number $\mu,$ whose coordinate functions are defined by
\begin{equation}\label{eq_genLSP}
F_j(X) = \phi_j(X) -b_j,  \quad \forall j =1,\ldots , k, 
\quad \forall X\in \mathcal{C}.
\end{equation}
We can summary this algorithm as follows.

\begin{alg} \label{alg0} \rm 
Find minimal-rank matrix solving problem (\ref{eq_Frmp_prob}).
%\hrulefill
\\
{\it Input}: Scalars $b_1,\ldots, b_k$ and function $\phi.$
\\	
{\it Output}: a solution $X\in \mathcal{C} \subset \R^{m\times n}$  to (\ref{eq_Frmp_prob}) .
	\begin{enumerate}
		\item 
			Set $r=1.$
		\item 
			Solve system (\ref{eq_genLSP}) by applying 
     		the Levenberg-Marquardt method \cite{r665}. 
		\item 
			If  (\ref{eq_genLSP}) has a numerical solution then stop.\\
	      	Else, set $r=r+1$ and go to Step 2.
%	\hrule	      	
	\end{enumerate}
\end{alg}

\vskip4mm
In fact, to perform the experiments, 
the variable matrices are vectorized.
Namely, 
the functions $F_j$  in (\ref{eq_genLSP}) is  
$\vect(X).$
This suggests us to study the rank one solutions to systems of equations.

\section{Affine rank minimization problem over arbitrary matrices} \label{sec_rmp_on_ArbMat}
In this section we are concentrating on numerically  solving ARM-problem (\ref{eq_rmp_prob}).
Using Proposition \ref{prop_CharLmap} each of the $k$ linear equations 
$\ell_i(X) = b_i$, $i = 1,2,\ldots,k$ is written as 
$$
\ell_i(X)= \tr(A_i^\mathrm{T}X) = \tr(A_iX^\mathrm{T}) = b_i.
$$
The function of the least square problem in this case is determined as:
$$
F_j(X) = \ell_j(X) - b_j, \quad j=1, \ldots , k.
$$
It is clear that such a matrix $X$ of rank $r$ can be found only first $r$ columns  
$X(:, 1:r)$
and its $n-r$ last ones are identified to zero vectors.
However,
we will see later this may not applicable in some particular cases,
for example, the matrix completion problem below.
A modification is to apply Proposition \ref{prop_rkMat} to find $X$ as 
$X = Y^\mathrm{T}Z$
for two matrix variables $Y\in\R^{r\times m}$ and $Z\in \R^{r\times n}.$
Namely we now focus on the following problem
\begin{equation}\label{eq_mrmp_prob}
\begin{array}{llll}
\mbox{minimize} & \rk(Y^\mathrm{T}Z) \\
\mbox{subject to}& \\
								&(Y,Z)\in \R^{r\times m}\times \R^{r\times n},\\
								& \ell(Y^\mathrm{T}Z) = b.
\end{array}
\end{equation}
Problem (\ref{eq_mrmp_prob}) is then a special case of problem (\ref{eq_Frmp_prob}) with $\phi(Y,Z) = \ell(Y^\mathrm{T}Z).$

To perform this modification, we need the following auxiliary results.
Set
$W = [ Y \quad Z]  \in \R^{r\times (m+n)}$
for each 
$(Y,Z)\in\R^{r\times m} \times \R^{r\times n}.$
Recall that the linear map
$\ell: \R^{m\times n} \rightarrow \R^k$
is defined by $k$ matrices
$A_1,\ldots, A_k \in \R^{m\times n}.$
That is
$$
\ell(U) =\left[\tr(A_1^\mathrm{T} U) \quad \ldots \quad   \tr(A_k^\mathrm{T} U) \right]^\mathrm{T}, \enskip \forall U\in \R^{m\times n}.
$$
For each $r=1,2,\ldots, p=\min\{m,n\},$
the least squares problem in this situation is then defined by the function 
$F: \R^{r\times (m+n)} \equiv \R^{r\times m} \times \R^{r\times n} \longrightarrow \R^k,$
%defined by
$$
F(W) = F(Y,Z) = \ell(Y^\mathrm{T}Z) -b, \quad \forall W= (Y,Z) \in  \R^{r\times m} \times \R^{r\times n}.
$$
It is clear that each coordinate function
$F_i: \R^{r\times (m+n)} \equiv \R^{r\times m} \times \R^{r\times n} \longrightarrow \R$
is determined by
$$
F_i(W) = F_i(Y,Z) = \tr(A_i^\mathrm{T}Y^\mathrm{T}Z) -b_i, \quad \forall W= (Y,Z) \in  \R^{r\times m} \times \R^{r\times n}.
$$
The Jacobian  matrix of $F$ can hence be calculated as
$$
\jac (F) 
= \frac{\partial F}{\partial W} 
= \frac{\partial \vect F}{\partial \vect W}
= \begin{bmatrix}
	\frac{\partial F_1}{\partial W} \\ 
	\vdots \\
	\frac{\partial F_k}{\partial W}  
\end{bmatrix}
= \begin{bmatrix}
	\frac{\partial F_1}{\partial Y} & \frac{\partial F_1}{\partial Z}  \\ 
	\vdots & \vdots \\
	\frac{\partial F_k}{\partial Y}  & \frac{\partial F_k}{\partial Z}
\end{bmatrix}
\in \R^{k\times r(m+n)},
$$
where
\begin{align*}
\vect(W)
&= [\vect(Y)^\mathrm{T} \quad \vect(Z)^\mathrm{T}]^\mathrm{T},\\
\frac{\partial F_i}{\partial Y} 
&=  \frac{\partial \tr(A_i^\mathrm{T} Y^\mathrm{T}Z) }{\partial Y}
= \vect(ZA_i^\mathrm{T})^\mathrm{T}, \\
\frac{\partial F_i}{\partial Z} 
&=  \frac{\partial \tr(A_i^\mathrm{T} Y^\mathrm{T}Z) }{\partial Z}
= \vect(YA_i)^\mathrm{T}.
\end{align*}
Algorithm \ref{alg0} will find a solution $W=[Y \quad Z]$ and then a solution to problem (\ref{eq_mrmp_prob}) can be defined as $X=Y^T Z.$
Some corresponding numerical results are presented in Section \ref{sec_numexp}.

\section{Rank minimization problem over positive semidefinite matrices} \label{sec_rmp_on_psdMat}
In this section we focus on the ARM-problem for semidefinite matrices (\ref{eq_Frmp_prob}).
By Proposition \ref{prop_CharLmap}, 
we can characterize the linear map $\ell$ by $k$ symmetric matrices 
$A_1,\ldots, A_k\in \Sy^n,$
with respect to 
$b_1,\ldots, b_k.$
In the subsection below, we develop some more properties of the solutions to a system of linear equations.
This might be not for our algorithm but it could be useful information in literature.

\subsection{Solutions to systems of linear equations}

Consider the system of linear equations as follows:
\begin{equation}\label{eq_mprob}
 \tr(A_i X)= b_i, \quad i=1, \ldots, m,
\end{equation}
where $A_i, X$ are real symmetric of order $n$
and $b=[b_1 \ \ldots \ b_m]^T \in \R^m.$
The corresponding homogeneous of system (\ref{eq_mprob}) is defined by
\begin{equation}\label{eq_hmprob}
 \tr(A_i X)= 0, \quad i=1, \ldots, m.
\end{equation}
On the other hand, for such a nonhomogeneous system (\ref{eq_mprob}), 
we call the system
\begin{equation}\label{eq_1mprob}
\tr(\tilde{A}_i \tilde{X} ) = 0, \quad i=1, \ldots , m,
\end{equation}
with
$
\tilde{A_i} =
\begin{bmatrix}
A_i & 0\\ 0 & -b_i
\end{bmatrix},\
%\tilde{X} \in \Sy^{n+1},
$
its 
``\textit{dominating system}''.

We also note that system (\ref{eq_mprob}) can be written in the classical form:
\begin{equation} \label{eq_ClasSyst}
\svec(A_i)^T \svec(X) = b_i,
\quad \mbox{or} \quad 
\mathcal{A} \svec (X) = b,
\end{equation}
where
$$ 
\mathcal{A} = 
\begin{bmatrix}
\svec(A_1)^T \\
\vdots\\
\svec(A_m)^T
\end{bmatrix} \in \R^{m\times \tau(n)}.
$$
Such a system has $\tau(n)$ variables. 
It is well known 
by Kronecker-Capelli theorem \cite{Kurosh80} 
that the system 
$\mathcal{A} \ \svec (X) =b$
has a solution if and only if 
$\rk(\mathcal{A}) = \rk (\tilde{A}),$
where $\tilde{\mathcal{A}} := [\mathcal{A}\ |\ b].$
Moreover, if 
$\rk(\mathcal{A}) = \rk (\tilde{A})= r$
then such a system has only one solution when $r= \tau(n).$
In the case $r< \tau(n),$
such a system has many solutions in which $r$ variables linearly dependent on $\tau(n)-r$ other variables.
We also note that the system $\{A_i\}$ is linearly (in)dependent in $\Sy^n$ if and only if 
so is the system $\{\svec(A_i)\}$ in $\R^{\tau(n)}.$

The following result gives us an equivalence of the non-homogeneous linear system  and 
a homogeneous linear system in one more variable,
in the case that the matrices $A_i$ are linearly independent.
\begin{prop}\label{prop_EquiExist}
System (\ref{eq_mprob}), with linearly independent matrices $A_i$'s and  $b\neq 0,$ $m\leq \tau(n),$
 has a solution (must be nonzero) if and only if system (\ref{eq_1mprob})
has a nontrivial solution.

Moreover,
if the positive semidefiniteness of a solution to one of these two systems is valuable
then so is a solution to the other system.
\end{prop}
\begin{proof}
If
$0\neq X =[x_{ij}]\in \Sy^n$   is a  solution to (\ref{eq_mprob})
then one can check
$$
0\neq 
\tilde{X} =
\begin{bmatrix}
X\ \ \ \  & 0_{n\times 1}\\
0_{1\times n} & 1\ \ \ \ 
\end{bmatrix}
\in \Sy^{n+1}
$$
is a solution to (\ref{eq_1mprob}) since
$$
\tr(\tilde{A}_i \tilde{X}) = \tr(A_i X)  - b_i = 0, \quad \forall i=1,\ldots, m.
$$

For the opposite direction,
we first note that if the homogeneous dominating system (\ref{eq_1mprob}) has nonzero solutions then
there exists one whose $(n+1, n+1)$st entry is nonzero. 
Indeed,
since $\{A_i\}_{i=1}^m$ is linearly independent, so is $\{\tilde{A}_i\}_{i=1}^m.$
But the homogeneous dominating system (\ref{eq_1mprob}) has $m$ equations and $\tau(n+1)$ variables.
Its solution vector space is hence of $\tau(n+1)-m>0$ dimensional since it has a nonzero positive semidefinite solution, provided by the hypothesis.
A basis vector can be chosen with $t:=\tilde{x}_{(n+1)(n+1)} \neq 0.$ 
Indeed,
if every solution $t$ was zeros then there would exist a nonsingular matrix $P$ (exists from the Gaussian elimination)  such that
$$
P \tilde{A} =
\begin{bmatrix}
\bullet  &\bullet & \ldots & \bullet\\
0  &\bullet & \ldots & \bullet\\
\vdots  &\vdots & \ddots & \vdots\\
0		  &0		 & \cdots & 1\\
\end{bmatrix} .
$$
This implies $\rk( \mathcal{A}) <m.$
This contradicts to the fact that $\{A_i\}_{i=1}^m $ is linearly independent.

With a solution $\tilde{X}$ satisfying the above discussion,
let $X$ be the $n\times n$ leading principle submatrix of $\tilde{X},$
we have
$$
0 = \tr(\tilde{A}_i \tilde{X}) = \tr(A_i X) -b_i t, \quad i=1,\ldots, m.
$$
This $\frac{1}{t} X$
is a solution of (\ref{eq_mprob}).

The rest of the proposition is an immediate consequence of what have shown above.
\end{proof}

\begin{rema}\rm
Even though some nonzero solutions of two systems (\ref{eq_mprob}) and $(\ref{eq_1mprob})$ stated in Proposition \ref{prop_EquiExist} simultaneously exist,
they do not need have the same rank. 
To see this, 
let us consider the linear system 
$\tr(A_1 X) = \tr(A_3 X) =0,$
$\tr(A_2X) = -1,$
where
$$
A_1 = \diag(1,-1,0), \quad
A_2 = \diag(1,0,-1) \mbox{ and }
A_3 = 
\begin{bmatrix}
0 & 1 & 0 \\
1 & 0 & 0\\
0 & 0 & 0\\
\end{bmatrix}
$$
and they are linearly independent.
The matrices $\tilde{A}_i$ are then defined as 
$$
\tilde{A}_1 = \diag(1,-1,0,0), \quad
\tilde{A}_2 = \diag(1,0,-1,-1) \mbox{ and }
\tilde{A}_3 = 
\small{
\begin{bmatrix}
0 & 1 & 0 & 0\\
1 & 0 & 0 & 0\\
0 & 0 & 0& 0\\
0 & 0 & 0 & 0
\end{bmatrix}.
}
$$
It is shown in \cite{Zhao2013} that the dominating homogeneous system (\ref{eq_1mprob}) has no rank-one solution 
but a rank-three solution 
$\tilde{X} = \diag(1,1,1,0).$
In our situation,
 we can find a rank-two solution, for example, 
$
\tilde{Y} = 
\scriptsize{
\begin{bmatrix}
0 & 0 & 0 & 0\\
0 & 0 & 0 & 0\\
0 & 0 & 1 & 1\\
0 & 0 & 1 & -1
\end{bmatrix}
}.
$
However, 
the initial non-homogeneous system defined by the matrices $A_1, A_2, A_3$ has a rank-one solution
$
X = 
\scriptsize{
\begin{bmatrix}
0 & 0 & 0 \\
0 & 0 & 0 \\
0 & 0 & 1
\end{bmatrix}.
}
$
\end{rema}

%\subsubsection{Conditions for the existence of rank-one  solutions}

%Given $m$ symmetric matrices $A_1, \ldots, A_m\in \Sy^n,$
The following proposition tells us the relationship between the existence of a positive definite element in 
$\texttt{Span}(A_1, \ldots, A_m) :=\{\sum\limits_{i=1}^m t_i A_i:\ t_i \in \R\}$
and that of trivial solution of the system $\tr(A_iX) =0,$ $i= \overline{1,m}$
over $\Sy_+^n.$
This is due to Bohnenblust \cite{Boh48} and is restated in some equivalent versions in 
\cite{r697, KlepSch13, Zhao2013}. 
In the their works, the proofs 
are mainly based on either the separation theorem for two nonempty convex sets (see, eg., \cite[Theorem III.1.2]{Barvinok02})
or the SDP duality theory (see, eg., \cite{b493}).
In our situation, 
we use only knowledge on linear algebra,
in particular, the theory of orthogonal complement in an inner-product vector space.
This also gives us a stronger result, compared with the existence one.

\begin{prop} {\rm \cite{KlepSch13, Zhao2013}}\label{prop_Boh}
With the notation above, we have
$$
\{X\in \Sy_+^n|\ \tr(A_iX) =0, \forall i=\overline{1,m}\}=\{0\}
\Longleftrightarrow
\Sy_{++}^n \cap \sp(A_1, \ldots, A_m) \neq \emptyset.
$$
%%%\begin{enumerate}[\rm i)]
%%%\item 
%%%	$\{X\in \Sy_+^n|\ \tr(A_iX) =0, \forall i=\overline{1,m}\}=\{0\}.$
%%%%\item
%%%%	$ \Sy_{++}^n \subseteq \sp(A_1, \ldots, A_m).$	 
%%%\item
%%%	$ \Sy_{++}^n \cap \sp(A_1, \ldots, A_m) \neq \emptyset.$	 	
%%%\end{enumerate}
%%%%System (\ref{eq_hmprob}) has only trivial solution $X=0$ is and only if
%%%%$\sp(A_1, \ldots, A_m)$ has a positive definite element.  
\end{prop}
%\begin{proof}
%\fbox{ii)$\Rightarrow $i)} If $A= \sum\limits_{i=1}^m t_i A_i \succ 0$ for some $m$-tuples $(t_1, \ldots, t_m)$ 
%and $X= \sum\limits_{j=1}^r x_j x_j^T,$ $x_j\in \R^n$
%is a solution to the system $\tr(A_i X) = 0,$ $i= \overline{1,m}$
%then for all $i=1, \ldots, m$ we have
%$$
%0 = \sum\limits_{i=1}^m t_i\tr(A_i X) = \sum\limits_{i=1}^m t_i \sum\limits_{j=1}^r  x_j^T A_ix_j
% = \sum_{j=1}^r x_j^T \left( \sum_{i=1}^m t_i A_i\right) x_j.
%$$
%The positive definiteness of $A$ implies $x_j=0$ for all $j=1, \ldots,r,$ and hence $X=0.$
%
%\fbox{i)$\Rightarrow $ii)} Conversely,
%we first note that $\Sy^n$ can be viewed as an inner-product vector space endowed with the inner product
%$$
%X\bullet Y := \tr(XY), \quad \forall X,Y\in \Sy^n.
%$$
%Since $\{A_i\}_{i=1}^m$ is linearly independent,
%the space of solutions to system (\ref{eq_hmprob}) in $\Sy^n$ has dimension 
%$d:= \tau(n) -m \geq 0.$
%Let 
%$\{\Delta_j\}_{j=1}^d$
%be an orthogonal basis of this solution space,
%where we agree the basis is empty if $d=0.$
%Here we note that $\Delta_j$'s are all neither positive nor negative semidefinite
%since system (\ref{eq_hmprob}) has only trivial solution.
%Because 
%$\tr(A_i \Delta_j) =0$
%for all $i=\overline{1,m}$ and $j=\overline{1,d},$
%$$
%\sp(A_1, \ldots, A_m)\oplus^\perp \sp(\Delta_1, \ldots, \Delta_d) = \Sy^n.
%$$  
%Since (\ref{eq_hmprob}) has only trivial solution,
%any positive definite matrix in $\Sy^n$ cannot belong to $\sp(\Delta_1, \ldots, \Delta_d).$
%It thus must be in $\sp(A_1, \ldots, A_n).$
%\end{proof}

We have already known by Proposition \ref{prop_ChoFac} i)
that any positive semideifnite matrix $X\in \Sy_+^n$ with $\rk(X)=r$ can be expressed as 
$X= \sum\limits_{i=1}^{r} x_i x_i^T$ for some $x_i\in \R^n,$ $i=1, \ldots, r.$
Based on the fact
$$
\tr(A_i^TX) = \tr(A_i \sum \limits_{i=1}^{r} x_jx_j^T) = \sum_{j=1}^r x_j^TA_i x_j, \quad \forall i=1, \ldots, m,
$$
the problem of finding a low-rank solution to (\ref{eq_mprob}) is of the form %of (\ref{eq_qmprob}):
\begin{equation}\label{eq_qmprob2}
 \hat{x}^T \widehat{A}_i \hat{x} = b_i, \quad i=1, \ldots, m,
\end{equation}
where the coefficient matrices now are 
$\widehat{A}_i:=\oplus_{j=1}^{r} A_j$ 
and 
$\hat{x} = [x_1^T \enskip \ldots \enskip x_r^T]^T.$
It is clear that a nonzero solution to (\ref{eq_qmprob2}) gives a solution to system (\ref{eq_mprob})
with rank less than or equal to $r.$ 
This is because of that $x_1, \ldots, x_r$ might be linearly dependent.
We thus have the following.
\begin{prop} \label{prop_r1-lr}
If system (\ref{eq_mprob}) has a solution of rank $r$ then system (\ref{eq_qmprob2}) has a nonzero solution.
Conversely,
if system(\ref{eq_qmprob2}) has a nonzero solution then system (\ref{eq_mprob}) has a solution of rank less than or equal to $r.$ 
\end{prop}

\begin{rema}\rm
i) When 
system (\ref{eq_mprob}), with $b\neq 0,$ has a nonzero positive semidefinite solution then by the work of Barvinok \cite{r662}, there is another positive solution with the rank at most
$
\dfrac{\sqrt{8m +1} -1}{2}.
$
This upper bound is smaller than $n$ since
$m\leq \tau(n).$
For us
this bound is sharpest by now.% in general.
%Namely, system (\ref{eq_mprob}) may has only solutions with rank exactly equal to this upper bound.
%Example \ref{ex_UpBound} below shows this fact.

For homogeneous system (\ref{eq_hmprob}), this bound does not necessary hold \cite{Zhao2013}.

ii) According to the works \cite{r605, HM2016}, one obtains 
$$
\max\left\{ \rk X:\ X\in \sp(A_1, \ldots, A_m) \right\} \geq \dfrac{2n+1 - \sqrt{(2n+1)^2 -8m} }{2}.
$$ 
This, indeed, follows from the proofs for lower bound in \cite{r605, HM2016}.
\end{rema}

%\begin{exam} \label{ex_UpBound}\rm 
%We consider the system of linear equations finding a Gram matrix of the sum-of-square bivariate polynomial
%$f_5$ which is constructed in
%\cite{b116}
%as follows:
%\begin{eqnarray*}
%f_1(x,y) &=& 1, \\
%f_i(x,y) &=&   \sum_{j=1}^i \prod_{k=2}^j\Delta_{2^{i-k}}^2(x,y) ,  \  i=2,3,\ldots,
%\end{eqnarray*}
%where 
%$
%\Delta_1(x,y) = y, 
%$
% and
%$$
%\Delta_r(x,y) = y \prod_{s=2}^r(y-sx)\in \Z[x,y], \ r = 2,3, \ldots.
%$$
%%
%It is shown in \cite{b116} that 
%$\deg(f_i)= 2^i-2,$ 
%and 
%it is theoretically proved that 
%its Pythagoras number is equal to $i,$
%for any $i=1,2,\ldots$ 
%When $i\geq 5,$ for example $i=5,$ the Pythagoras number of $f_5$ is $5.$
%On other words,
%its Gram matrix has rank at least $5.$
%We also have
%$
%\frac{\sqrt{8m +1} -1}{2} =5,
%$ 
%where $m=\deg(f_5)+1=17$ and $n=\frac{\deg(f_5)}{2}+1=9.$
%Here we note that the order of the Gram matrices of $f_5$ equals to 
%the degree of polynomials in sum-of-square terms of $f_5$ plus 1.
%\end{exam}

\subsection{Algorithm}
Even though this is a particular case of the ARM-problem over arbitrary matrices,
Proposition \ref{prop_ChoFac}
allows us to find a Cholesky factor $Y \in \R^{n \times r}$ instead of a positive semidefinite matrix,
and
this leads to a reduction in number of variables for the ARM-problem.
So problem (\ref{eq_Frmp_prob}) can be cast in the following form
\begin{equation}\label{eq_CholRMP_prop}
\begin{array}{llll}
\mbox{minimize} & \rk(Y) \\
\mbox{subject to}& \\
%								&X\in \Sy_+^n,\\
								& [\tr(A_1^\mathrm{T} YY^\mathrm{T})\quad\ldots \quad \tr(A_k^\mathrm{T} YY^\mathrm{T})]^\mathrm{T}=\ell(YY^\mathrm{T}) = b. 
\end{array}
\end{equation}
The idea for solving this problem is similar to the previous case,
where one checks whether there exists a matrix with  lowest possible rank satisfying the requirements.
In this situation,
at the step corresponding to $r,$
the following function is applied:
$F: \R^{n\times r} \longrightarrow \R^k$
defined by
$$
F(Y) = \ell(YY^\mathrm{T}) -b, \quad \forall Y \in  \R^{n\times r} .
$$
The coordinate functions
$F_i: \R^{n\times r}  \longrightarrow \R^k$
are obviously defined by
$$
F_i(Y) =  \tr(A_i^\mathrm{T}YY^\mathrm{T}) -b_i, \quad \forall Y \in  \R^{n\times r}.
$$
The Jacobian  matrix of $F$ in this case follows from (\ref{eq_JacTr1}):
$$
\jac (F) 
= \frac{\partial F}{\partial Y} 
= \frac{\partial \vect F}{\partial \vect Y}
= \begin{bmatrix}
	\frac{\partial F_1}{\partial Y} \\ 
	\vdots \\
	\frac{\partial F_k}{\partial Y}  
\end{bmatrix}
%= \begin{bmatrix}
%	\frac{\partial F_1}{\partial Y} & \frac{\partial F_1}{\partial Z}  \\ 
%	\vdots & \vdots \\
%	\frac{\partial F_k}{\partial Y}  & \frac{\partial F_k}{\partial Z}
%\end{bmatrix}
\in \R^{k\times rn},
$$
where
\begin{align*}
\frac{\partial F_i}{\partial Y} 
&=  \frac{\partial \tr(A_i^\mathrm{T}Y Y^\mathrm{T}) }{\partial Y}
= \vect[(A_i^\mathrm{T}+A_i)Y]^\mathrm{T}.
\end{align*}

\section{Applications} \label{sec_app}

In this section, we consider three applications 
of problem (\ref{eq_Frmp_prob}).

\subsection{Low-rank matrix completion}
In machine learning scenarios,
e.g., in factor analysis, collaborative filtering, and latent semantic
analysis \cite{Rennie2005, Srebro2004,RFP10},
there are several problems that can be reformulated as the low-rank matrix completion problem.
Given the values of some entries of a matrix, 
this problem fills the missing entries of the matrix 
such that its rank is small as possible.
This problem is summarized and reformulated as follows.
Given a set of triples
$$
(R,C,S) \in \{1,\ldots, m\}^k \times \{1,\ldots, n\}^k \times \R^k, 
$$
and we wish to construct a small-as-possible rank matrix 
$X=[X_{rs}]\in \R^{m\times n},$ 
such that 
$X_{R(i), C(i)} = S(i) $
for all $i=1,\ldots,p.$
This can be reformulated as
\begin{equation}\label{eq_mc}
\begin{array}{llll}
\mbox{minimize} & \rk(X) \\
\mbox{subject to}& \\
								& X_{R(i), C(i)} = S(i), \quad \forall i=1,\ldots, p. 
\end{array}
\end{equation}
This problem can then be  solved by using Algorithm \ref{alg0}.

\subsection{Low-dimensional Euclidean embedding problems}

Euclidean distance matrices, shortly EDMs, have received 
increased attention because of its many applications
which can be found in
eg., \cite{RFP10,Dattorro05, Parhizkar13, BG05} and references there in.

We first recall this problem.
Let 
$D=[d_{ij}]\in \Sy^n$ 
be  a
{\it Euclidean distance matrix} 
(EDM)
associated to the points 
$x_1,\ldots, x_n \in \R^r,$
i.e.,
\begin{equation} \label{eq_edm_cond}
d_{ij} = \|x_i-x_j \|^2 = x_i^\mathrm{T}x_i + x_j^\mathrm{T} x_j - 2 x_i^\mathrm{T} x_j, \quad i,j= 1,\ldots, n.
\end{equation}
The smallest positive integer number $r$ is said to be the
\textit{embedding dimension} of $D.$

Following \cite{RFP10},
let
$
\mathbf{1} \in \R^{n\times 1} 
$
be the column vector of ones.
Define
$
V \triangleq I_n -\frac{1}{n} \mathbf{1}\mathbf{1}^\mathrm{T}.
$
Note that 
$V$ 
is the orthogonal projection matrix onto the hyperplane 
$\{v\in \R^{n\times 1}:\ \mathbf{1}^\mathrm{T} v = 0\}.$
In particular,
$$\rk(V) = n-1$$
and $V$ has an eigenvector $\mathbf{1}$ with respect to the eigenvalue zero.
It follows from the work in \cite{Schoenberg35} that
$D$
is an EDM of $n$ points in $\R^r$
if and only if three following conditions hold:
\begin{align*}
d_{ii}=0,& \quad\forall i=1,\ldots,n;\\
- VDV \succeq 0; & \\
\rk(VDV) \leq r. 
\end{align*}
Given a positive integer number $n$ and 
\textit{partial Euclidean matrix}
$D_0,$ 
i.e.,
every entry of $D_0$ is either ``specified'' or ``unspecified'', $\diag(D_0) = 0,$ 
and
every fully specified principal sub-matrix of $D_0$ is also a Euclidean distance matrix. 
The 
\textit{low-dimensional Euclidean embedding problem} 
finds 
a Euclidean matrix $D$ consistent with the known pairwise distances described by $D_0$ 
and associated to a number of points in the smallest dimensional space $\R^r.$
Such a problem can be reformulated as the ARM-problem \cite{RFP10}
\begin{equation}\label{eq_edm_prob}
\begin{array}{llll}
\mbox{minimize} & \rk(VDV) \\
\mbox{subject to}& \\
								&- VDV \succeq 0 , \\
								&\quad \ell(D) = b, 
\end{array}
\end{equation}
where $\ell: \Sy^n \rightarrow \R^p$ is 
an appropriate linear map, corresponding to the specified entries in $D_0$,
including the condition that makes the diagonal of $D$ to be zero.
If one sets
$X = [x_1 \ldots x_n] \in \R^{r\times n}$
then $D$ can be found in form
\begin{equation}
D=\mathcal{D} (X) := \diag(X^\mathrm{T} X) \mathbf{1}^\mathrm{T} + \mathbf{1} \diag(X^\mathrm{T}X)^\mathrm{T} - 2 X^\mathrm{T}X
\end{equation}
because of  (\ref{eq_edm_cond}).
Since $V \mathbf{1} =0,$ 
\begin{equation}
- V\mathcal{D}(X) V  = 2 VX^\mathrm{T} XV.
\end{equation}
Substituting this fact into problem (\ref{eq_edm_prob}) we get the equivalent one:
\begin{equation}\label{eq_edm_prob2}
\begin{array}{llll}
\mbox{minimize} & \rk(XV) \\
\mbox{subject to}& \\
								& X \in \R^{r \times n},\\
								& \ell(\mathcal{D}(X)) = b.
\end{array}
\end{equation}
%\red{The map $\ell(\mathcal{D}(X))$ is not linear in $X$ anymore!}
It is clear that the above problem is of the form of problem (\ref{eq_Frmp_prob}) with 
$\phi(X) = \ell(\mathcal{D}(X)).$

We make the 
linear map $\ell$ more explicit as in \cite{KW12}.
Let $H$ be the $1$-$0$ adjacency matrix, i. e.,
$$
h_{ij} =
\left\{
\begin{array}{lll}
1 & \mbox{if} & (i,j) \in E,\\
0 & \mbox{if} & (i,j) \not\in E,
\end{array}
\right.
$$  
for the set of subscripts $E$ corresponding to the specified entries of $D_0$.
The main problem is to find an as-small-as-possible rank completion $D$ of $D_0.$
Namely,
one needs to find $D$ in the form
\begin{equation} \label{eq_D-Z}
\begin{array}{rll}
D &=& \diag(Z) \mathbf{1}^\mathrm{T} + \mathbf{1} \diag(Z)^\mathrm{T} - 2 Z,\\
Z &=& X^\mathrm{T} X,\\
H\odot D &=& H\odot D_0,
\end{array}
\end{equation}
where $\odot$ denotes the component-wise (or Hadamard) matrix product.
With the help of the fact $\rk(XV) \leq \rk(X),$
problem (\ref{eq_D-Z}) is then reduced to the rank minimization problem in the form
\begin{equation} \label{eq_D-Y}
\begin{array}{lrll}
\mbox{minimize} & \rk(X) \\
\mbox{subject to} \\ 
%&\mathcal{D}(X) &=& \diag(X^\mathrm{T}X) \mathbf{1}^\mathrm{T} + \mathbf{1} \diag(X^\mathrm{T}X) - 2 X^\mathrm{T}X,\\
%& X^\mathrm{T}X \mathbf{1} &=& \mathbf{0},\\
& X \in \R^{r \times n},\\
&H\odot \mathcal{D}(X) &=& H\odot D_0.
\end{array}
\end{equation}
%\red{Problem (\ref{eq_D-Y}) is not in the form (\ref{eq_Frmp_prob}) anymore but formulating it in $Z$ it is again in this form?}
In many applications one seeks the embedding dimension being two or three.

Note that when $D$ is determined as $D = \mathcal{D}(X),$
the following formula is used to compute the Jacobian matrices used in the Levenberg-Marquardt 
algorithm:
$$
\dfrac{\partial d_{ij}}{\partial X} =
\left[
\underbrace{0 \ldots 0}_{r \mbox{ \scriptsize times} } 
\ldots 
\underbrace{0 \ldots 0}_{r \mbox{ \scriptsize times} }
\enskip
2 (x_i^\mathrm{T} - x_j^\mathrm{T})
\enskip
0
\ldots
0
\enskip
2 (x_j^\mathrm{T} - x_i^\mathrm{T})
\enskip
0 \ldots 0
\right]
\in \R^{1\times nr}.
$$
Here we assume $i\leq j.$

\section{Numerical experiments} \label{sec_numexp}

\subsection{General RM-problem and quadratic systems}

From the theoretical point of view,
the rank minimization problem is NP-hard so that 
there has not been any method directly solve this one in the literature.
A good way for solving the RM-problem  over positive semidefinite matrices  is to solve the corresponding  problem
that minimize the nuclear norm (see, eg., \cite{RFP10,Ma11}).
The nuclear norm minimization problem  (NNM-problem) 
is a really good one to give suitable lower and upper bounds for the 
original RM-problem.
Additionally, 
in \cite{RFP10} it is proved that
the NNM-problem over general matrices is tractable to solve since
it can be reformulated as a semidefinite program \cite{FHB01}.
Another method for solving the NNM-problem over positive semidefinite matrices
was proposed in \cite{Ma11} by using 
Modified
Fixed Point Continuation Method.

We now illustrate the RM-problem over generic matrices.
Table \ref{table_xrmp} shows the results 
obtained by Algorithms \ref{alg0} for this case.
The matrices $A_1,\ldots, A_k$ are randomly chosen with entries in $(0,1).$
The backward errors are  determined by
$$
\mathrm{err} = \frac{\| \ell(X)-b \|_2}{\|b\|_2}.
$$
The result for each case is averagely taken per three experiments. 
The numerical results show that 
Algorithm \ref{alg0} gives better solutions if the factorization in Proposition \ref{prop_rkMat} is applied.
More precisely,
we see in Table \ref{table_xrmp}, 
the results when Proposition \ref{prop_rkMat} is applied have smaller rank.
Table \ref{table_xrmp} also exhibits a comparison  between our method and the one described in \cite{RFP10}.

The NNM-problem approximating problem (\ref{eq_rmp_prob}) is followed in \cite{RFP10} and 
can be summarize as follows
\begin{equation} \label{eq_nnmp}
\begin{array}{llll}
\mbox{minimize} & \| X\|_* \\
\mbox{subject to} \\ 
 &\ell(X) = b,
\end{array}
\end{equation}
where $\|.\|_*$ denotes the nuclear norm of $X,$
which is the sum of all its singular values.
If $X$ has a singular value decomposition 
$X = U \Sigma V^\mathrm{T}$ then
one can solve problem (\ref{eq_nnmp}) by solving the semidefinite program:
\begin{equation} \label{eq_nnmp_SDP}
\begin{array}{llll}
\mbox{minimize} & \frac{1}{2}( \tr(W_1) + \tr(W_2) ) \\
\mbox{subject to} \\ 
& \begin{bmatrix} W_1 & X \\  X^\mathrm{T} & W_2 \end{bmatrix} \succeq 0,\\
 &\ell(X) = b.
\end{array}
\end{equation}
This is nice formulation in theoretical point of view but 
in practice the resulting matrices  may have ``high-rank'' by SDP solvers . 
One can see in Table \ref{table_xrmp}, 
where problem (\ref{eq_nnmp_SDP}) is implemented in \textsf{CVX} toolbox \cite{b470} of \textsc{Matlab} calling \textsf{Sedumi} \cite{Sturm99}, 
that 
the semidefinite program seems to give resulting matrices with full rank 
and less accuracy.
\begin{table}%[ht]
\centerline{
\begin{tabular}{lll|| l|| l || l } 
\\
 $m$  & $n$ &  $k$ & \begin{tabular}{cc} $X(:,1:r)$ \\ \hline $\rk(X)$ & $\mathrm{err}$   \end{tabular} & 
 \begin{tabular}{lr} $X = Y^T Z$ \\ \hline $\rk(X)$ & $\mathrm{err}$   \end{tabular}  
 &  \begin{tabular}{lr} SDP  \\ \hline $\rk(X)$ & $\mathrm{err}$   \end{tabular}
 %$m$  & $n$ &  $k$ & \rk(X)  
 \\  \hline 
 5       & 6		& 4		& 1 \quad \qquad 5.31e-16 &	1 \quad \qquad 2.44e-16	& 5 \quad \qquad \ 1.02e-09  			\\
 51     & 50	   & 51 	& 1 \quad \qquad  7.13e-15& 	1 \quad \qquad 4.49e-15 &50 \quad \qquad 8.42e-10  			\\
 50     & 100  & 81    & 2  \quad \qquad  5.56e-16&	1 \quad \qquad	 9.42e-16 & 50 \quad \qquad 
 3.38e-09\\%	& 150 	& 150 &  100 &  			\\
 50     & 200  & 100  &   3 \qquad\quad 5.46e-16&	1 \qquad\quad 7.31e-16	& 50 \quad \qquad 
 4.08e-09\\%	& 150 	& 150 &  100 &  			\\
 100   & 200  & 300  &   3   \qquad \quad 2.63e-14& 2   \qquad \quad	3.83e-15
 & \mbox{out of memory}\\%& 400 	& 400 &  350 &  \\
 500   & 550  & 300  &   1  \qquad \quad 1.66e-15& 1		\qquad \quad 1.93e-14	
 & \mbox{out of memory}\\%& 500 	& 500 &  450 &  \	
 500 	& 500 &  450 &  	1 \qquad \quad	   2.04e-15 & 1		\qquad \quad 1.10e-14
 & \mbox{out of memory}
\end{tabular}
}
\caption{Comparison between LM-method and SDP 
solving the RM-problem over $m\times n$ matrices.}  \label{table_xrmp}  
% SEE RMPROB/ncln_prob.m
\end{table}

For the RM-problem over positive semidefinite matrices,
the experiments perform with randomly chosen symmetric matrices
$A_1,$ 
$\ldots,$ 
$A_k.$
The result for each case is also averagely taken per three experiments. 

Table \ref{table_srmp} shows 
a comparison between our method and the one described in \cite{Ma11}.
The errors in this table are computed as
$$
\mathrm{err} = \frac{\|\ell(X) -b\|_2}{\|b\|_2}.
$$
\begin{table}[ht]
\centerline{
\begin{tabular}{c  c | c | c}
\\
$n$ &  $k$ & 
\begin{tabular}{cc} \rk(X)  &  \\  \hline  LM    &  AFPC-BB \end{tabular} &  
\begin{tabular}{cc} err & \\  \hline  LM      & \qquad AFPC-BB \end{tabular}
\\ \hline \hline
100      & 579   &
\begin{tabular}{cr} 
 6		 & \qquad  10 			
\end{tabular}
&
\begin{tabular}{cc} 
 1.89e-16		& \quad 9.46e-4 			
\end{tabular}
\\
%======================================================
200      & 1221   &
\begin{tabular}{cc} 
 7		&  		\qquad 10	
\end{tabular}
&
\begin{tabular}{cc} 
 1.87e-15		&\quad 9.84e-4
\end{tabular}
\\
%======================================================
500      & 5124   &
\begin{tabular}{cc} 
 11		&  		\qquad 10	
\end{tabular}
&
\begin{tabular}{cc} 
 2.52e-15		&\quad 4.90e-3
\end{tabular}
\\
%======================================================
500      & 3309   &
\begin{tabular}{cc} 
 7		&  		\qquad 27	
\end{tabular}
&
\begin{tabular}{cc} 
 3.00e-15	&\qquad \quad \mbox{ NA }
\end{tabular}
%%%%%%
\end{tabular}
}
\caption{Comparison between LM-method and AFPC-BB 
solving the RM-problem over positive semidefinite matrices. 
%\red{Last entry should be filled in.}
The error of the method AFPC for the case $(n,k)=(500, 3309)$ 
is not shown in \cite{Ma11}}.
 \label{table_srmp}
\end{table}
What we see in Table  \ref{table_srmp} that the values of the rank of resulting matrices obtained by our method are 
smaller the ones obtained by solving the corresponding NNM-problem.

Table \ref{table_lrmc} shows the results for several values of $m,n.$
We  take $R,C\in \N^k $  with the entries are random in 
$\{1,\ldots, m\},$ 
$\{1,\ldots, n \},$
respectively,
and so is 
$S\in (0,1)^k.$
For the cases $m=n,$ it turns out the results for the systems of quadratic equations.
More precisely,
the system has solution if the solutions' have rank one.

\begin{table}[ht]
\centerline{
\begin{tabular}{lllc || lllc} 
\\
 $m$  & $n$ &  $k$ & \rk(X) & $m$  & $n$ &  $k$ & \rk(X)  \\  \hline 
 5       & 6		& 4		& 4 			&50    & 50    &  51 & 		2	\\
 51     & 50	   & 51 	&  3  			& 100 	& 100 &  50 &     1\\
 50     & 100  & 81    &   1  			& 150 	& 150 &  100 &  	1		\\
 50     & 200  & 100  &    1			& 200 	& 200 &  200 &       2\\
 100   & 200  & 300  &     3 			& 400 	& 400 &  350 &       1\\
 500   & 550  & 300  &     1 			& 500 	& 500 &  450 &       1\	
\end{tabular}
}
\caption{Solution to the low-rank matrix completion using LM-method.}\label{table_lrmc}
\end{table}

\subsection{Euclidean distance matrix problem}
This section shows the numerical results for problem (\ref{eq_D-Y}).
All tests are dealt with partial Euclidean matrices $D_0$ with entries randomly taken in the interval $[0,1].$

Table \ref{table_edm} shows the Euclidean embedding dimensions for all cases that $D_0$ are dense, i.e., 
the entries of the corresponding matrix $H$ are all one.
\begin{table}[h]  %%%%See  m-file: edmcp.m
\centerline{
\begin{tabular}{lcc || lcc} 
\\
  $n$ & \rk(X) & err   				&  $n$ & \rk(X)  & err\\  \hline 
4		&  2			& 5.09e-16         &100    & 2		& 1.33e -14			\\
10 	&  2			&	9.38e-16			& 150 	& 2			&1.99e-14  \\
20    &   2  		&	2.09e-15			& 200 	& 2			&2.54e-14 			\\
30  & 2  			&4.04e-15 			& 300 	&  2			& 3.84e-14   \\
40  & 2   			&6.10e-15			& 400 	&  2        & 4.20e-14\\
50  & 2   			&	5.98e -15		& 500 	&  2       &  4.32e-14	
\end{tabular}
}
\caption{Solution to EDM problems with respect to dense partial EDM matrices.}\label{table_edm}
%SEE RMPROB/edmcp.m
\end{table}

Table \ref{table_edm2} shows the  results for sparse matrices $D_0,$  i.e., 
the entries of the corresponding matrix $H$ are either zero or one.
\begin{table}[h]  %%%%See  m-file: edmcp.m
\centerline{
\begin{tabular}{lcc || lcc} 
\\
  $n$ & \rk(X) & err   				&  $n$ & \rk(X)  & err\\  \hline 
4		&  2			& 1.95e-16         &100    & 2		& 4.03e -16			\\
10 	&  2			&	5.16e-16			& 150 	& 2			&3.97e-16  \\
20    &   2  		&	4.56e-16			& 200 	& 2			&4.23e-16 			\\
30  & 2  			&4.85e-16 			& 300 	&  2			& 3.89e-16   \\
40  & 2   			&4.15e-16			& 400 	&  2			& 4.42e -16    \\
50  & 2   			&	4.00e -16		& 500 	&  2			& 4.54e -16    \	
\end{tabular}
}
\caption{Solutions to EDM problems with respect to randomly-chosen sparse partial EDM matrices.}\label{table_edm2}
%SEE RMPROB/edmcp.m
\end{table}

The backward errors of
all tests in both cases of $D_0$
 are determined as
$$
\mbox{err } = \dfrac{\| D - D_0\|_2}{\|D_0\|_2}.
$$
It turns out that the configurations of our experiments are all in two dimensional spaces.

\section{Conclusion and discussion}
We have proposed  an algorithm for solving the rank minimization problem over a subset of $\R^{m\times n}$ 
determined by a differentiable function. 
As a consequence, the affine rank minimization problems over either arbitrary or positive semidefinite matrices have been numerically tested.
This  algorithm was then applied to 
solve the low-rank matrix completion problem and  
the low-dimensional Euclidean embedding problem.
Some numerical experiments have been performed to illustrate our algorithms as well as the applications.

We have also developed some useful properties for low rank solutions to systems of linear matrix equations.
This suggests us a reformulation of
the IIR and FIR low-pass filter problems 
described in \cite{LeVaBa16}
as optimization problems over rank-one  positive semidefinite matrices.
In the future we will deal with this method to solve such filter design problem.
This might be suitable because the resulting positive semidefinite matrices derived by SDP solvers in \cite{LeVaBa16}
are usually full rank.
Obviously, this requires a much more amount of memory and complexity in comparison with rank-one setting.

\vskip4mm
\noindent
\textbf{Acknowledgement.}
The author would like to thank Prof. Marc Van Barel for 
his valuable discussion which led to improvements in the manuscript and
the \textsc{Matlab} codes.

% % 
 % \bibliographystyle{plain}
%  \bibliography{RMProblem}

\begin{thebibliography}{10}

\bibitem{r697}
A.~Barvinok,
\newblock {\em A Remark on the Rank of Positive Semidefinite Matrices Subject to
  Affine Constraints},
\newblock {Discrete Comput. Geom.}  \textbf{25} (2001), no. 1, 23--31.

\bibitem{Barvinok02}
A.~Barvinok.
\newblock {\em A Course of Convexity}, volume~54 of {\em Graduate Studies in
  Mathematics}.
\newblock American Mathematical Society, Providence RI, 2002.

\bibitem{r662}
A.~I. Barvinok.
\newblock {\em Problems of distance geometry and convex properties of quadratic
  maps}.
  \newblock {Discrete Comput. Geom.}  \textbf{13} (1995), 189--202.


\bibitem{Bhatia07}
R.~Bhatia.
\newblock {\em Positive Definite Matrices}.
\newblock Princeton Series in Applied Mathematics. Princeton University Press,
  2007.

\bibitem{Boh48}
F.~Bohnenblust.
\newblock {\em Joint positiveness of matrices}.
\newblock Technical report, California Institute of Technology, 1948.
\newblock Avaliable at
  \texttt{\scriptsize{http://orion.uwaterloo.ca/~hwolkowi/henry/book/fronthandbk.d/Bohnenblust.pdf}}.

\bibitem{BG05}
I.~Borg and P.~J.~F. Groenen.
\newblock {\em {Modern Multidimensional Scalling: Theory and Applications}}.
\newblock Springer Science+Business Media, Inc., 2005.

\bibitem{r605}
M.-D. Choi, T~Y Lam, and B~Reznick,
\newblock {Sums of squares of real polynomials},
\newblock in {\em Proceedings of Symposia in Pure Mathematics}, vol.~58,
1995, pp. 103--126.

\bibitem{Dattorro05}
J.~Dattorro,
\newblock {\em Convex optimization $\&$ Euclidean distance geometry},
\newblock Meboo, 2005.

\bibitem{Fazel02}
M.~Fazel,
\newblock {\em Matrix rank minimization with applications},
\newblock PhD thesis, Stanford University, March 2002.

\bibitem{FHB01}
M.~Fazel, H.~Hindi, and S.~Boyd,
\newblock {A Rank Minimization Heuristic with Application to Minimum Order
  System Approximation},
\newblock In {\em Proceedings of the American Control Conference, Arlington,
  VA}, 2001, pp. 4734--4739.

\bibitem{b470}
M.~Grant and S.~P. Boyd,
\newblock {\em \textsf{CVX}: Matlab software for disciplined convex programming,
  version 2.1}, 
  June 2015.
\newblock \url{http://cvxr.com/cvx}.

\bibitem{b398}
R.~A. Horn and C.~R. Johnson,
\newblock {\em Matrix analysis},
\newblock Cambridge University Press, 1985.

\bibitem{KlepSch13}
I.~Klep and M.~Schweighofer,
\newblock {\em An Exact Duality Theory for Semidefinite Programming Based on Sums
  of Squares}.
\newblock {\em Math. Oper. Res.} \textbf{38} (2013), no. 3, 569--590. 

\bibitem{KW12}
N.~Krislock and H.~Wolkowicz,
\newblock {Euclidean Distance Matrices and Applications}.
\newblock In M.~F. Anjos and J.~B. Lasserre, editors, {\em Handbook on
  Semidefinite, Conic and Polynomial Optimization},  pp. 879--914,
  Internat. Ser. Oper. Res. Management Sci., 166,
  Springer, New York, 2012.

\bibitem{Kurosh80}
A.~G. Kurosh,
\newblock {\em Higher Algebra (English translation from Russian by G.
  Yankovsky)}.
\newblock Moscow Mir Publishers, 1980.

\bibitem{LeSoVa13}
T.~H. Le, L.~Sorber, and M.~{Van Barel},
\newblock {The Pythagoras number of real sum of squares polynomials and sum of
  square magnitudes of polynomials},
\newblock {\em Calcolo} \textbf{50} (2013), no. 4, 283--303. 

\bibitem{LeVaBa16}
T.~H. Le and M.~{Van Barel},
\newblock {A convex optimization model for finding non-negative polynomials}.
\newblock {\em J. Comput. Appl. Math.} \textbf{301} (2016), 121--134. 

\bibitem{HM2016}
T.~H. Le and M.~{Van Barel},
\newblock {On bounds of the Pythagoras number of the sum of square magnitudes of
  laurent polynomials}.
\newblock {\em Numer. Algebra Control Optim.} \textbf{6} (2016), no. 2, 91--102. 

\bibitem{Ma11}
Y.~Ma and L.~Zhi,
\newblock {The Minimum-Rank Gram Matrix Completion via Modified Fixed Point
  Continuation Method},
\newblock in: {\em ISSAC 2011--Proceedings of the 36th International Symposium on Symbolic and Algebraic Computation}, 
241--248, ACM, New York, 2011. 

\bibitem{MP97}
M.~Mesbahi and G.~P. Papavassilopoulos,
\newblock {On the rank minimization problem over a positive semidefinite linear
  matrix inequality}.
\newblock {\em IEEE Trans. Automat. Control} \textbf{42} (1997), no. 2, 239--243. 

\bibitem{Parhizkar13}
R.~Parhizkar,
\newblock {\em Euclidean Distance Matrices: Properties, Algorithms and
  Applications},
\newblock PhD thesis, \'{E}cole polytechnique f\'{e}d\'{e}rale de Lausanne,
  2013.

\bibitem{q715}
K.~B. Petersen and M.~S. Pedersen,
\newblock {\em The matrix cookbook},
\newblock Technical University of Denmark, 2012.
\newblock \url{http://www2.imm.dtu.dk/pubdb/p.php?3274}.

\bibitem{b116}
A.~Prestel and C.~N. Delzell.
\newblock {\em Positive Polynomials}, volume~53 of {\em Springer monographs in
  mathematics}.
\newblock Springer, 2013.

\bibitem{RFP10}
B.~Recht, M.~Fazel, and P.~A. Parrilo,
\newblock {Guaranteed minimum-rank solutions of linear matrix equations via
  nuclear norm minimization},
\newblock {\em SIAM Rev.} \textbf{52} (2010), no. 3, 471--501. 

\bibitem{Rennie2005}
J.~D.~M. Rennie and N.~Srebro,
\newblock Fast maximum margin matrix factorization for collaborative
  prediction,
\newblock in: {\em Proceedings of the 22nd International Conference of Machine
  Learning}, 713--719, 2005.

\bibitem{Schoenberg35}
I.~J. Schoenberg,
\newblock {Remarks to Maurice Fr\'{e}chet's article ``Sur la d\'{e}finition
  axiomatique d'une classe d'espaces distanci\'{e}s vectoriellement applicable
  sur l'espace de Hilbert''},
\newblock{\em Ann. of Math.} \textbf{38} (1935), no. 3, 724--732. 


\bibitem{r665}
L.~Sorber, M.~{Van Barel}, and L.~{De Lathauwer},
\newblock {Unconstrained optimization of real functions in complex variables}.
\newblock {\em SIAM J. Optim.} \textbf{22} (2012), no. 3, 879--898. 

\bibitem{Srebro2004}
N.~Srebro,
\newblock {\em Learning with Matrix Factorizations},
\newblock Phd thesis, Massachusetts Institute of Technology, Cambridge, MA,
  2004.

\bibitem{Sturm99}
J.~F. Sturm.
\newblock {Using SeDuMi 1.02, a \textsc{Matlab} toolbox for optimization over
  symmetric cone}.
\newblock {\em Optim. Methods Softw.} \textbf{11/12} (1999), no. 1-4, 625--653. 

\bibitem{p878}
L.~Vandenberghe and S.~P. Boyd,
\newblock {Semidefinite programming},
\newblock {\em SIAM Rev.}\textbf{38}(1996), 49--95.

\bibitem{b493}
H.~Wolkowicz, R.~Saigal, and L.~Vandenberghe,
\newblock {\em Handbook of semidefinite programming: theory, algorithms and
  applications},
\newblock Kluwer Academic Publishers, 2000.

\bibitem{Zhao2013}
Y.~B. Zhao and M.~Fukushima,
\newblock Rank-one solutions for homogeneous linear matrix equations over the
  positive semidefinite cone,
\newblock {\em Appl. Math. Comput.} \textbf{219} (2013), no. 10, 5569--5583. 

\end{thebibliography}
%\end{document}

  \end{document}